\documentclass{aomart}
\usepackage{amssymb,amsmath,mathrsfs,amsthm, mathtools}
\usepackage[english]{babel}
\usepackage{tikz-cd}
\usepackage{enumitem}
\usepackage{comment}
\usepackage{tcolorbox}

\makeatletter
\newcommand*{\arabiceven}[1]{%
  \expandafter\@arabiceven\csname c@#1\endcsname
}
\newcommand*{\@arabiceven}[1]{%
  \@arabic{\numexpr(#1)*2\relax}%
}
\AddEnumerateCounter\arabiceven\@arabiceven{0}
\makeatother

\def\N{\mathbb{N}}

\def\R{\mathbb{R}}

\let\inclu\hookrightarrow

\def\id{\mathrm{id}}

\def\set{\mathsf{Set}}
\def\top{\mathsf{Top}}
\def\ctd{\mathsf{Centered}}
\def\rst{\mathsf{Raster}}
\def\ftb{\mathsf{Filterbase}}
\def\pretop{\mathsf{PreTop}}
\def\cat{\mathsf{C}}

\newcommand{\card}{\mathop{\text{card}}}

\def\P{\mathcal{P}}
\def\up{\uparrow}

\newtheorem{defn}{Definition}[section]

\newtheorem{prop}{Proposition}[section]

\newtheorem{rmk}[defn]{Remark}

\title[The emergence of filters in topological categories]{The emergence of the concept of filter in topological categories}
\author{Giacomo Dossena}
\email{giacomo.dossena@gmail.com}
\copyrightnote{August 2015}

\begin{document}

\begin{abstract}
In all approaches to convergence where the concept of filter is taken as primary, the usual motivation is the notion of neighborhood filter in a topological space. However, these approaches often lead to spaces more general than topological ones, thereby calling into question the need to use filters in the first place. In this note we overturn the usual view and take as primary the notion of convergence in the most general context of centered spaces. In this setting, the notion of filterbase emerges from the concept of germ of a function, while the concept of filter emerges from an amnestic modification of the subcategory of centered spaces admitting germs at each point.
\end{abstract}

\maketitle
\thispagestyle{plain}

\tableofcontents

\section{Introduction}\label{sec:intro}

Filters on a set have been introduced in 1937 by Henri Cartan (see \cite{cartan1937a, cartan1937b}) to replace sequences in the study of a general topological space (if the space is not first-countable sequences are not enough). From the point of view of order theory, a filter on a set is dual to an ideal in the lattice of subsets: in this sense filters have been introduced independently the same year by Marshall Stone under the name of $\mu$-ideals (see \cite{stone1938}). Compare also \cite{riesz1908} and \cite{ulam1929} for earlier instances of the notion of a filter on a set\footnote{It should be noted that Conditions 1, 2, 3 on p.23 of \cite{riesz1908} actually axiomatize the notion of ultrafilter, making the maximality condition (Condition 4 on the same page) superfluous.}.

Filters are also employed to define certain generalizations of the notion of a topological space, for instance pretopologies, pseudotopologies, convergences (see e.g.~\cite{col-low2001} for a review of these notions), which in turn provide new instructive and significant ways to think of the concept of a topological space. However, to the best of the author's knowledge, in all such instances filters are introduced without much motivation except their role in topology, which makes their use in defining generalizations of topological spaces look less compelling. It is the purpose of this note to offer a possibly new way to arrive at the concept of a filter (and filterbase) by means of certain generalized topological spaces.

In the next section we recall three different types of collections of subsets that we are going to use in the following: rasters, filterbases and filters. In section \ref{sec:coincidence} we relate filters and filterbases on a set $X$ to certain equivalence relations on the set of functions on $X$. In section \ref{sec:centered} we introduce the category of centered spaces and, by applying results from section \ref{sec:coincidence}, we select the subcategory of spaces admitting germs at each point. Finally, section \ref{sec:category} studies the mutual relationships among several topological concrete categories arising by using rasters, filterbases or filters as neighborhood systems, finally showing that the category which uses filters arises as an amnestic modification of the category that uses filterbases.

\section{Rasters, filterbases, filters}\label{sec:collections}
In this section we recall the definition of filter and filterbase, along with the much less known notion of raster (introduced in \cite{richmond-slapal2010}). These notions determine three different kinds of collections of subsets of a given set. While rasters and filterbases are not comparable in general, filters instead are precisely those collections which are simultaneously a filterbase and a raster.

Given a nonempty set $X$ and a nonempty collection $\P$ of subsets of $X$, let us consider the following three conditions.
\begin{enumerate}[label=(\textnormal{F}\arabic{*}), ref=(\textnormal{F}\arabic{*}), start=0]
\item\label{defn:filter0} $\forall n\in\N$, $A_1,\dots, A_n\in\P\implies A_1\cap\cdots\cap A_n\neq\emptyset$,
\item\label{defn:filter1} $A\supset B\in \P\implies A\in\P$,
\item\label{defn:filter2} $A,B\in\P\implies \text{there is } C\in \P$ such that $C\subset A\cap B$.
\end{enumerate}

\begin{defn}
Let $X$ and $\P$ be as above. Then:
\begin{displaymath}
\P \text{ is a }\begin{cases}\text{raster}\\\text{filterbase}\\\text{filter}\end{cases}\text{ on $X$ if it satisfies Conditions }\begin{cases}\ref{defn:filter0},\ref{defn:filter1}\\\ref{defn:filter0},\ref{defn:filter2}\\\ref{defn:filter0},\ref{defn:filter1}, \ref{defn:filter2}\end{cases}.
\end{displaymath}
\end{defn}

Examples of filters are: the collection of all supersets of a given nonempty subset $A\subset X$; the collection of all cofinite subsets of an infinite set $X$ (also called Fr\'{e}chet filter); the collection of all neighborhoods of a point of a topological space. An example of a filterbase which is not a filter is the collection of balls of radius $1/n$, $n\in\N$, centered at a given point of $\R^k$. An example of a raster which is not a filter is the collection $\{\{0,1\},\{1,2\},\{0,1,2\}\}$ of subsets of $\{0,1,2\}$.

\section{Coincidence sets}\label{sec:coincidence}
In this section we relate filters and filterbases on $X$ to certain equivalence relations on the set of functions on $X$.

Let $X$ and $Y$ be two nonempty sets and let $Y^X$ denote the set of all functions from $X$ to $Y$. Fix a nonempty collection $\P$ of subsets of $X$. By means of $\P$ we define a binary relation $\hat{\P}$ on $Y^X$ as follows.

\begin{defn}\label{defn:relation1}
Given $f,g\in Y^X$ we say $f$ is $\P$-related to $g$, and we write $f \hat{\P} g$, if the coincidence set of $f$ and $g$ belongs to $\P$. In other words, $f \hat{\P} g$ if $\{x\in X\mid f(x)=g(x)\}\in\P$.
\end{defn}

To save typographic space, we sometimes write $\{f=g\}$ for $\{x\in X\mid f(x)=g(x)\}$. Notice that the relation $\hat{\P}$ is symmetric because ``$=$'' is symmetric.

\begin{prop}\label{prop:filter}
If $\card{Y} \geq 3$, then $\P$ is a filter if and only if $\hat{\P}$ is a nontrivial equivalence relation.
\end{prop}
\begin{proof}
Let us first assume that $\P$ is a filter on $X$ and prove that $\hat{\P}$ is a nontrivial equivalence relation (we will not need the cardinality assumption on $Y$ for this implication). $\hat{\P}$ is reflexive if and only if $X\in\P$, which holds true for $\P$ because of \fullref{Condition}{defn:filter1} in the definition of filter. Now we show that $\hat{\P}$ is transitive. Assume $f\hat{\P}g$ and $g\hat{\P}h$ for some $f,g,h\in Y^X$. Then $\{f=h\}\supset \{f=g\} \cap \{g=h\}$ and since $\P$ is a filter we get $\{f=h\}\in\P$, that is, $f\hat{\P}h$. Therefore $\hat{\P}$ is an equivalence relation. Moreover, it is nontrivial, otherwise $\emptyset\in\P$ which is forbidden by \fullref{Condition}{defn:filter0}. Now let us assume that $\hat{\P}$ is a nontrivial equivalence relation and prove that $\P$ is a filter on $X$. Pick three distinct points $y_1, y_2, y_3\in Y$ and for each $A,B\subset X$ define the following functions:
\begin{equation}
\begin{aligned}
f(x)&\coloneqq		\begin{cases}
				y_1 & x\in A\cap B,\\
				y_2 & x\in (A\cup B)\setminus (A\cap B),\\
				y_3 & \text{otherwise};
				\end{cases}\\
g(x)&\coloneqq		\begin{cases}
				y_1 & x\in A\cup B,\\
				y_2 & \text{otherwise};\\
				\end{cases}\\
h(x)&\coloneqq		\begin{cases}
				y_2 & x\in A\setminus B,\\
				y_1 & \text{otherwise}.
				\end{cases}
\end{aligned}
\end{equation}
It can be checked that:
\begin{equation}\label{eq:AB}
\begin{aligned}
A&=\{x\in X\mid f(x)=h(x)\},\\
B&=\{x\in X\mid h(x)=g(x)\},\\
A\cap B&=\{x\in X\mid f(x)=g(x)\}.
\end{aligned}
\end{equation}
Now assume $A,B\in\P$. By \eqfullref{Equations}{eq:AB} and the definition of $\hat{\P}$ we have $f \hat{\P} h$ and $h \hat{\P} g$, and by transitivity we get $f \hat{\P} g$, that is, $A\cap B\in\P$. Now assume $B\in\P$ and take any $A\supset B$, so that $A\cap B=B$. By \eqfullref{Equations}{eq:AB} we have $f \hat{\P} g$ and $g \hat{\P} h$, and by transitivity we obtain $f \hat{\P} h$, that is, $A\in\P$. This proves Conditions \ref{defn:filter1} and \ref{defn:filter2} for $\P$. Since $\hat{\P}$ is nontrivial, $\emptyset\not\in\P$ and thus $\P$ is a filter.
\end{proof}

\begin{rmk}
The implication ``if $\P$ is a filter then $\hat{\P}$ is a nontrivial equivalence relation'' is implicitly used in the construction of the hyperreal field ${}^*\R$ of non-standard analysis: ${}^*\R$ is obtained as the set of $\hat{\P}$-equivalence classes in $\R^\N$ where the filter $\P$ on $\N$ is a free ultrafilter.
\end{rmk}

\begin{rmk}
The condition $\card{Y}\geq 3$ is sharp. Take $X=\{0,1,2,3\}$, $Y=\{0,1\}$, $\P=\{\{0\}\cup\{i\}\mid i=1,2,3\}$. Then $\P$ is not a filter (not even a filterbase!) but $\hat{\P}$ is a nontrivial equivalence relation, as can be (tediously) checked.
\end{rmk}

\begin{rmk}
\fullref{Proposition}{prop:filter} is not new: the author is quite sure to have read it somewhere. Unfortunately he could not find any reference in the literature (any help in recovering a reference is appreciated). There is also a striking affinity of \fullref{Proposition}{prop:filter} to a circle of ideas involving ultrafilters, monads, Arrow's theorem, and more. A good starting point is \cite{leinster2013}. To give a flavour of the affinity, consider the following characterization of an ultrafilter (Proposition 1.5 in \cite{leinster2013}): ``Let $X$ be a set and $\P\subset 2^X$. Then $\P$ is an ultrafilter if and only if, whenever $X$ is expressed as a disjoint union of three subsets, exactly one belongs to $\P$''. The number $3$ cannot be lowered.
\end{rmk}

We now enlarge the relation in \fullref{Definition}{defn:relation1} by weakening it as follows.

\begin{defn}\label{defn:relation2}
Given $f,g\in Y^X$ we say $f$ is weakly $\P$-related to $g$, and we write $f \hat{\P}_w g$, if the coincidence set of $f$ and $g$ contains an element of $\P$. In other words, $f \hat{\P}_w g$ if there is $P\in\P$ such that $P\subset \{x\in X\mid f(x)=g(x)\}$.
\end{defn}

As before, $\hat{\P}_w$ is symmetric because ``$=$'' is symmetric. Moreover, $\hat{\P}_w$ is reflexive because $\P$ is nonempty.

\begin{prop}\label{prop:filterbase}
If $\card{Y} \geq 2$, then $\P$ is a filterbase if and only if $\hat{\P}_w$ is a nontrivial equivalence relation.
\end{prop}
\begin{proof}
Let us first assume $\P$ is a filterbase on $X$ and prove that $\hat{\P}_w$ is a nontrivial equivalence relation (as in \fullref{Proposition}{prop:filter}, we will not need the cardinality assumption on $Y$ for this implication). Symmetry and reflexivity are settled, so let us show that $\hat{\P}_w$ is transitive. Assume $f\hat{\P}_w g$ and $g\hat{\P}_w h$ for some $f,g,h\in Y^X$, which means there are $P_1,P_2\in\P$ such that $P_1\subset \{f=g\}$ and $P_2\subset  \{g=h\}$. By \fullref{Condition}{defn:filter2} there is $P_3\in\P$ such that $P_3\subset P_1\cap P_2$, but $P_1\cap P_2\subset\{f=g\}\cap\{g=h\}\subset\{f=h\}$, therefore $P_3\subset\{f=h\}$ which means $f\hat{\P}_w h$. Clearly $\hat{\P}_w$ is nontrivial, otherwise $\emptyset\in\P$ which is forbidden by \fullref{Condition}{defn:filter0}. Now let us assume $\hat{\P}_w$ is a nontrivial equivalence relation and prove that $\P$ is a filterbase on $X$. Pick two distinct points $y_1, y_2\in Y$ and for each $A,B\subset X$ define the following functions:
\begin{equation}
\begin{aligned}
f(x)&\coloneqq	\begin{cases}
			y_1 & x\in B,\\
			y_2 & \text{otherwise};
			\end{cases}\\
g(x)&\coloneqq	\begin{cases}
			y_1 & x\in A\cup B,\\
			y_2 & \text{otherwise};\\
			\end{cases}\\
h(x)&\coloneqq	\begin{cases}
			y_2 & x\in B\setminus A,\\
			y_1 & \text{otherwise}.
			\end{cases}
\end{aligned}
\end{equation}
It can be checked that:
\begin{equation}\label{eq:AB'}
\begin{aligned}
A&=\{x\in X\mid g(x)=h(x)\},\\
B&\subset\{x\in X\mid f(x)=g(x)\},\\
A\cap B&=\{x\in X\mid f(x)=h(x)\}.
\end{aligned}
\end{equation}
Now assume $A,B\in\P$, so that $g\hat{\P}_w h$ and $f\hat{\P}_w g$. By transitivity we have $f\hat{\P}_w h$, that is, there is $C\in\P$ such that $C\subset \{f=h\}=A\cap B$. This proves \fullref{Condition}{defn:filter2} for $\P$. Since $\hat{\P}_w$ is nontrivial, $\emptyset\not\in\P$ and thus $\P$ is a filterbase.
\end{proof}

\begin{rmk}
The condition $\card{Y}\geq 2$ is sharp: if $\card{Y}=1$ then the set $Y^X$ is a singleton, so the only possible equivalence relation is the trivial one.
\end{rmk}

\section{Centered spaces}\label{sec:centered}
In this section we introduce a very general notion of ``space'' and apply the results from the previous section to select a more specific kind of space involving filterbases.

\begin{defn}[Centered spaces]
A \textbf{centered structure} on a set $X$ is given by assigning to each point $x\in X$ a collection $\nu(x)$ of subsets of $X$ with the property that $x\in N$ for each $N\in\nu(x)$. A set together with a centered structure is called a \textbf{centered space}.
\end{defn}

\begin{rmk}
Centered structures (and morphisms between them, as defined below) appeared already in \cite{csaszar2002} where they are called \emph{generalized neighborhood systems}. We opted for a perhaps less evocative but shorter name.
\end{rmk}

\begin{rmk}
The above definition is evidently inspired by the notion of a neighborhood of a point in a topological space. Analogously, a centered structure allows us to define convergence of sequences: a sequence $s\colon\N\to X$, $n\mapsto s_n$, in the centered space $(X,\nu)$ is said to converge to $x\in X$ if each $N\in\nu(x)$ contains a tail\footnote{A tail of the sequence $s$ is a subset of $X$ of the form $\{s_n\mid n\geq k\}$ for some $k\in\N$.} of $s$. With this definition, the requirement that for every $x\in X$ each $N\in\nu(x)$ contains the point $x$ is equivalent to the (very natural and desirable) requirement that each constant sequence converges.
\end{rmk}

We now define the notion of a morphism between centered spaces. We interpret $\nu(x)$ as a system of ``probes'' allowing us to localize the point $x$. If we are given two systems $\nu_1(x)$ and $\nu_2(x)$ at the same point $x$, we say that $\nu_1(x)$ is finer than $\nu_2(x)$, and we write $\nu_1(x)\preceq \nu_2(x)$, if for each $N\in\nu_2(x)$ there is $N'\in\nu_1(x)$ such that $N'\subset N$. Given a centered space $(X,\nu)$ and a set $Y$, a function $f\colon X\to Y$ allows us to transport $\nu(x)$ from $x$ to $f(x)\in Y$ by simply taking the image $f(\nu(x))\coloneqq\{f(N)\mid N\in\nu(x)\}$. The idea behind the next definition is then to consider only those functions between centered spaces that produce finer (or, to say it better, \emph{not coarser}) centered structures after transportation.

\begin{defn}[Morphisms between centered spaces]
A function $f\colon X\to Y$ between two centered spaces $(X,\nu_X)$ and $(Y,\nu_Y)$ is called \textbf{centered at $x\in X$} if $f(\nu_X(x))\preceq\nu_Y(f(x))$. A function which is centered at each $x\in X$ will be simply called centered.
\end{defn}

Evidently, every topological space can be considered as a centered space by taking as $\nu(x)$ the collection of neighborhoods of $x$, and the notion of centeredness for functions thus becomes continuity. We shall explore the relation between topological spaces and centered spaces more thoroughly in the next section.

Now let $(X,\nu_X)$ and $(Y,\nu_Y)$ be two centered spaces, and fix a point $x\in X$. We consider the relation of being weakly $\nu_X(x)$-related, as defined in \fullref{Definition}{defn:relation2}, restricted to the set $C_x(X,Y)$ of functions from $X$ to $Y$ which are centered at $x$. For ease of notation, let us introduce the symbol $\hat{x}$ for this restricted relation.

In other words, $f\hat{x}g$ if $f$ and $g$ are centered at $x$ and they agree on some element of $\nu_X(x)$. The definition is analogous to the well known construction of the germ at $x$ of a function $f$ on a topological space as the equivalence class of functions agreeing with $f$ on some neighborhood of $x$, with the difference that here the neighborhoods are substituted by the more general collection $\nu_X(x)$. Indeed, in the topological case this relation is an equivalence. By noting that in the proof of \fullref{Proposition}{prop:filterbase} the functions $f,g,h$ are centered at $x$ whatever the centered structure on\footnote{In fact, they have been carefully chosen precisely for this reason.} $Y$, we can state the following key result.

\begin{prop}\label{prop:filterbase2}
If $\card{Y} \geq 2$, then $\nu_X(x)$ is a filterbase if and only if being $\hat{x}$-related is a nontrivial equivalence relation on $C_x(X,Y)$.
\end{prop}

The upshot of the previous proposition is that, if we want to be able to speak of germs of functions at $x$ in a centered space, it is necessary and sufficient that $\nu(x)$ be a filterbase. In this sense, the concept of filterbase corresponds to the existence of germs in centered spaces.

\begin{rmk}
Notice that the previous proposition remains true if we consider the weak $\nu_X(x)$-relation to be defined on the whole set of functions from $X$ to $Y$, not just those which are centered at $x$. On the other hand, it is very likely that one implication becomes false if we restrict to everywhere centered functions (or functions which are centered on some $N\in\nu_X(x)$). The problem is to be found in the ``lack of coherence'' among different $\nu(x)$s when varying $x\in X$, if $(X,\nu)$ is a generic centered space. What remains true is the other implication: if $\nu(x)$ is a filterbase, whatever set of functions on $X$ we choose, the weak $\nu_X(x)$-relation will always be an equivalence.
\end{rmk}

\section{Categorical considerations}\label{sec:category}

For our next purposes it is convenient to introduce two operations on collections of subsets. Given a collection $\P$ of subsets of $X$, we define:
\begin{equation}
\begin{aligned}
\P^\up\coloneqq &\{A\subset X\mid \exists P\in\P\text{ such that }P\subset A\}\\
\P^\cap\coloneqq &\{A\subset X\mid (\exists n\in\N)(\exists P_1,\dots,P_n\in\P)\text{ such that }A=\cap_{i=1}^n P_i\}
\end{aligned}
\end{equation}
Note the formulas
\begin{itemize}
\item $\P\subset\P^\up$,
\item $\P\subset\P^\cap$,
\item $\P^{\up\up}=\P^\up$,
\item $\P^{\cap\cap}=\P^\cap$,
\item $\P^{\up\cap}=\P^{\cap\up}$.
\end{itemize}

\begin{prop}
Given $\P\in 2^{2^X}$ with $\emptyset\not\in\P$, the following hold:
\begin{itemize}
\item $\P$ is a raster iff $\P=\P^\up$,
\item $\P$ is a filterbase iff $\P\preceq\P^\cap$,
\item $\P$ is a filter iff $\P=\P^{\cap\up}$,
\item $\P^\up$ is the smallest raster containing $\P$,
\item $\P^\cap$ is the smallest filterbase containing $\P$,
\item $\P^{\up\cap}$ is the smallest filter containing $\P$,
\item if $\P$ is a raster, then $\P^\cap$ is the smallest filter containing $\P$,
\item if $\P$ is a filterbase, then $\P^\up$ is the smallest filter containing $\P$.
\end{itemize}
\end{prop}
\begin{proof}
It is enough to apply the definitions.
\end{proof}

We are now ready to introduce several concrete categories of generalized topological spaces. They all share the same general construction of centered spaces: a generalized space is a set $X$ together with a choice, for each point of $X$, of a collection of subsets of $X$ containing the point $x$. Morphisms are defined as in centered spaces. Depending on the kind of collections that we allow, we obtain the following different categories: $\ctd$, $\rst$, $\ftb$, $\pretop$. In $\ctd$ no restriction is put on the collections $\nu(x)$ for each $x\in X$; in $\rst$ they are all rasters; in $\ftb$ they are all filterbases; in $\pretop$ they are all filters.

\begin{rmk}
Raster spaces are called neighborhood spaces in \cite{kent-min2002} (the original paper where they were introduced), in \cite{richmond-slapal2010} and in references therein. Since the latter name is also used in the literature to indicate pretopological spaces, to avoid confusion we opted for the new name ``raster spaces'' in accordance with the notion of raster introduced in \cite{richmond-slapal2010}.
\end{rmk}

There is an obvious diagram of full embeddings (the first embedding corresponds to the presentation of a topology in terms of neighborhood filters subject to a certain coherence condition):

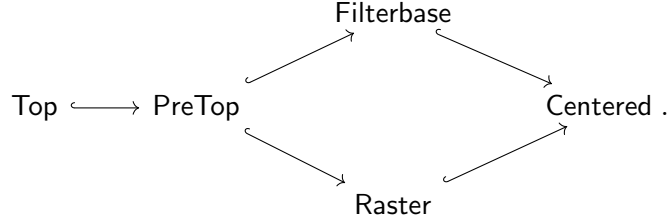
\begin{center}
\begin{tikzcd}
&	&	\ftb\arrow[hookrightarrow]{rd}	&	\\
\top\arrow[hookrightarrow]{r}   &    \pretop\arrow[hookrightarrow]{ru}\arrow[hookrightarrow]{rd}	&	&	\ctd\;.\\
&	&	\rst\arrow[hookrightarrow]{ru}	&
\end{tikzcd}
\end{center}

Each of these concrete categories is topological, as can be seen explicitly: given a set $X$ and a family $$\{(X_i,\nu_i)\mid i\in I\}$$ of centered, or raster, or filterbase, or pretopological, or topological spaces with functions $f_i\colon X\to X_i$, the corresponding initial structure on $X$ is determined by defining, for each $x\in X$,
\begin{equation}
\begin{aligned}
\nu(x)\coloneqq&\{f_i^{-1}(N)\mid N\in\nu_i(f_i(x)), i\in I\}&\text{ (in $\ctd$)}\\
\nu(x)\coloneqq&\{f_i^{-1}(N)\mid N\in\nu_i(f_i(x)), i\in I\}^\cap&\text{ (in $\ftb$)}\\
\nu(x)\coloneqq&\{f_i^{-1}(N)\mid N\in\nu_i(f_i(x)), i\in I\}^\up&\text{ (in $\rst$)}\\
\nu(x)\coloneqq&\{f_i^{-1}(N)\mid N\in\nu_i(f_i(x)), i\in I\}^{\cap\up}&\text{ (in $\pretop$ and $\top$)}
\end{aligned}
\end{equation}

We recall that for a concrete category $(\cat,|~|)$, a concrete subcategory $\cat'\subset \cat$ is concretely reflective in $\cat$ if for each $C\in\cat$ there is $C'\in\cat'$ with $|C|=|C'|$ such that the identity function $\id_{|C|}\colon C\to C'$ is a morphism, and every morphism from $C$ into an object of $\cat'$ factorizes through $\id_{|C|}$. The dual notion of concretely coreflective subcategory is obtained by reversing all arrows, as usual.

\begin{prop}
The various embeddings are concretely reflective (r) and/or concretely coreflective (c) according to the labels in the following diagram:
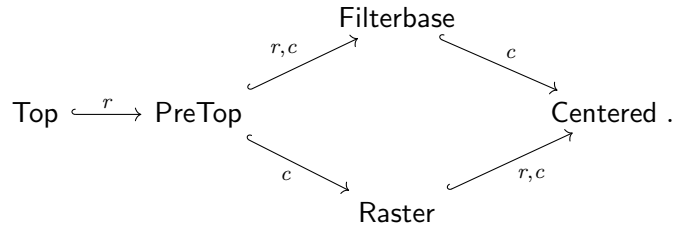
\begin{center}
\begin{tikzcd}
&	&	\ftb\arrow[hookrightarrow, "\text{c}"]{rd}	&	\\
\top\arrow[hookrightarrow,"\text{r}"]{r}   &    \pretop\arrow[hookrightarrow, "\text{r,c}"]{ru}\arrow[hookrightarrow, "\text{c}", swap]{rd}	&	&	\ctd\;.\\
&	&	\rst\arrow[hookrightarrow, "\text{r,c}", swap]{ru}	&
\end{tikzcd}
\end{center}
\end{prop}
\begin{proof}
Reflectivity of $\top\inclu\pretop$ is well known (given $(X,\nu)\in\pretop$, the reflection is $\id_X\colon(X,\nu)\to (X,\tau)$ where $\tau$ is a topology on $X$ defined by $\tau\coloneqq\{U\subset X\mid \forall x\in U,\;U\in\nu(x)\}$). For $\pretop\inclu\ftb$ the reflection is given by $\id_X\colon(X,\nu)\to(X,\nu^\up)$ where $\nu^\up(x)\coloneqq(\nu(x))^\up$ (immediate to check) and the coreflection is given by $\id_X\colon(X,\nu^\up)\to(X,\nu)$ (immediate to check). For $\ftb\inclu\ctd$ the coreflection is given by  $\id_X\colon(X,\nu^\cap)\to(X,\nu)$ where $\nu^\cap(x)\coloneqq(\nu(x))^\cap$. To check that this is a coreflection is less immediate, so we give more details: let $(X',\nu')\in\ftb$ and let $f\colon(X',\nu')\to(X,\nu)$ be a morphism in $\ctd$. We show that $f\colon(X',\nu')\to(X,\nu^\cap)$ is a morphism as well. Indeed, take $M\in\nu^\cap(x)$, which can be written as $M=N_1\cap\cdots\cap N_n$ for some $n\in\N$ and some $N_1,\dots,N_n\in\nu(x)$ by the definition of $\nu^\cap$. By assumption, for each $N_i$ there is $N_i'\in\nu'(x)$ such that $f(N_i')\subset N_i$, therefore $f(\cap_i N_i')\subset\cap_i N_i=M$ and since $\nu'(x)$ is a filterbase then there is $N''\in\nu'(x)$ such that $N''\subset \cap_i N_i'$. Thus we have $f(N'')\subset f(\cap_i N_i')\subset M$. For $\pretop\inclu\rst$ the coreflection is given by $\id_X\colon(X,\nu^\cap)\to(X,\nu)$ and it can be checked in the same way as for the coreflectivity of $\ftb\inclu\ctd$ just proved. Finally, for $\rst\inclu\ctd$ the reflection is given by $\id_X\colon(X,\nu)\to(X,\nu^\up)$ and the coreflection is given by $\id_X\colon(X,\nu^\up)\to(X,\nu)$ (immediate to check).
\end{proof}

\begin{prop}
The embeddings $\pretop\inclu\ftb$ and $\rst\inclu\ctd$ are categorical equivalences.
\end{prop}
\begin{proof}
It is enough to notice that for each $(X,\nu)\in\ftb$ the function $\id_X\colon(X,\nu)\to(X,\nu^\up)$ is an isomorphism in the category $\ftb$, and for each $(X,\nu)\in\ctd$ the function $\id_X\colon(X,\nu)\to(X,\nu^\up)$ is an isomorphism in the category $\ctd$.
\end{proof}

There is a more significant way to look at the above equivalences\footnote{See \cite{porst1994} for a discussion about the notion of ``concrete equivalence''.}. We recall that in a concrete category $(\cat,|~|)$ the fiber over $X\in\set$ is the collection of $\cat$-objects $C$ such that $|C|=X$. Each fiber comes equipped with a preorder defined as follows: for objects $C_1$, $C_2$ in the fiber over $X$, we write $C_1\leq C_2$ if $\id_X\colon C_1\to C_2$ is a morphism. The concrete category $\cat$ is called amnestic if, for each $C_1,C_2\in\cat$, $C_1\leq C_2$ and $C_2\leq C_1$ imply $C_1=C_2$. In other words, $\cat$ is amnestic if each equivalence class determined by the preorder contains exactly one object or, which is the same, the preorder just defined is actually a (partial) order. Informally, this means that ``each fiber does not have too many objects floating around'' (quotation taken from \cite{acc}). Every concrete category has a so-called amnestic modification, obtained by choosing one member from each equivalence class. Quoting again from \cite{acc}: ``With respect to almost every interesting categorical property, a concrete category is indistinguishable from its amnestic modification''. An amnestic modification is equivalent to the original category and is unique, up to concrete isomorphism.
\begin{prop}\label{prop:amnestic}
$\pretop$ is an amnestic modification of $\ftb$, and $\rst$ is an amnestic modification of $\ctd$.
\end{prop}
\begin{proof}
Let $\P_1$ and $\P_2$ be rasters (or filters). Then $\P_1\preceq\P_2$ and $\P_2\preceq\P_1$ imply $\P_1=\P_2$. 
\end{proof}

\section{Conclusions}
We showed that the concepts of filter and filterbase can be seen to emerge from the realm of topological concrete categories once we specify the notion of convergence in the context of centered spaces. Specifically:
\begin{itemize}
\item By Prop.~\ref{prop:filterbase2}, the notion of filterbase selects the subcategory $\ftb$ of $\ctd$ consisting of all centered spaces that admit germs at each point;
\item By Prop.~\ref{prop:amnestic}, the notion of filter arises from an amnestic modification of $\ftb$.
\end{itemize}

\end{document}